\documentclass{article} 
\usepackage{amssymb,amsmath,amsthm}
\usepackage{amscd}
\usepackage{hyperref}
\usepackage{epsfig}
\usepackage{amsfonts}
\usepackage{amssymb}
\usepackage{amsmath,enumerate}
\usepackage{commath}
\usepackage{euscript}
\usepackage{enumitem}
\usepackage{amscd}
\usepackage{xcolor}
\usepackage[ruled,vlined]{algorithm2e}
\usepackage[numbers,sort&compress]{natbib}

\newtheorem{theorem}{Theorem}[section]
\newtheorem{remark}{Remark}[section]

\newtheorem{definition}{Definition}[section]

\newtheorem{lemma}{Lemma}[section]

\numberwithin{equation}{section}

\begin{document}

	\title{Finite Convergence Criteria for Normalized Nash Equilibrium through Weak Sharpness and Linear Conditioning}
	\author{
		Asrifa Sultana\footnotemark[2] , Shivani Valecha\footnotemark[1]~\footnotemark[2] }
	\date{ }
	\maketitle
	\def\thefootnote{\fnsymbol{footnote}}
	
	\footnotetext[1]{ Corresponding author. e-mail- {\tt shivaniv@iitbhilai.ac.in}}
	\noindent
	\footnotetext[2]{Department of Mathematics, Indian Institute of Technology Bhilai, Raipur - 492015, India.
	}
	\maketitle
	\vspace*{-0.6cm}

	\vskip 4mm {\small\noindent {\bf Abstract.}
		The generalized Nash equilibrium problems play a significant role in modeling and analyzing several complex economics problems. In this work, we consider jointly convex generalized Nash games which were introduced by Rosen. We study two important aspects related to these games, which include the weak sharpness property for the set of normalized Nash equilibria and the linear conditioning technique for regularized Nikaido-Isoda function. Firstly, we define the weak sharpness property for the set of normalized Nash equilibria and then, we provide its characterization in terms of the regularized gap function. Furthermore, we provide the sufficient conditions under which the linear conditioning criteria for regularized Nikaido-Isoda function becomes equivalent to the weak sharpness property for the set of normalized Nash equilibria. We show that an iterative algorithm used for determining a normalized Nash equilibrium for jointly convex generalized Nash games terminates finitely under the weak sharpness and linear conditioning assumptions. Finally, we estimate the number of steps needed to determine a solution for the specific type of jointly convex generalized Nash game as an application.
		
		
		}

	\section{Introduction}
	
	The concept of Nash equilibrium initiated by John F. Nash, was defined for non-cooperative games consisting of finite number of players \cite{nash}. In Nash equilibrium problems (NEP), the feasible strategy set for each player is assumed to be fixed. Later on, Arrow-Debreu \cite{debreu} extended NEP and developed the notion of generalized Nash equilibrium problem (GNEP) in which the strategy sets of players are allowed to rely on the strategies selected by rivals. In recent times, the concept of GNEP has attracted considerable attention due to its applicability in electricity markets \cite{Aussel_Asrifa}, microeconomics \cite{cotrina}, environmental sustainability issues \cite{salas} and many other fields. For historical background on GNEP and a brief description of recent developments in this field, one can refer to an article by Facchinei-Kanzow \cite{facc}.
	
	Let us recall the formal definition of GNEP from \cite{debreu}. Consider the set $\{1,2,\cdots N\}$ containing $N$-players participating in a non-cooperative game. Suppose that player $i$ controls a strategy vector $x_i\in \mathbb{R}^{n_i}$ where $\sum_{i=1}^{N} n_i=n$. Assume that the vector $x=(x_i)_{i=1}^N=(x_{-i},x_i)\in \mathbb{R}^n$ indicates the strategies of all $N$-players, where $x_{-i}\in \mathbb{R}^{n-n_i}$ is strategy vector for all players except $i$. For any strategy vector $x_{-i}$ of rivals, the feasible strategy set of the player $i$ is restricted to $K_i(x_{-i})\subseteq \mathbb{R}^{n_i}$ and he intends to choose the strategy $x_i\in K_i (x_{-i})$ which is solution of the following problem $P_i(x_{-i})$ given as,
	\begin{equation}
		P_i(x_{-i}):\quad \displaystyle\min_{z_i}~\theta_i(x_{-i},z_i)~\text{ subject to}~z_i\in K_i(x_{-i}),\label{Nash}
	\end{equation} 
	considering that $\theta_i:\mathbb{R}^{n}\rightarrow\mathbb{R}$ indicates the loss function of player $i$.
	If the symbol $Sol_i(x_{-i})$ indicates the solution set for $P_i(x_{-i})$, then the vector $\bar x= (\bar x)_{i=1}^N$ is known as \textit{generalized Nash equilibrium} if  $\bar x_i\in Sol_i(\bar x_{-i})$ for each $i=1,2,\cdots N$.
	
	In one of the preliminary works by Rosen \cite{rosen}, an important class of GNEP is discussed, which has been recently termed as ``jointly convex generalized Nash equilibrium problem" \cite{facchinei}. For a given strategy vector $x_{-i}$ of the rivals, the feasible strategy set $K_i(x_{-i})$ of player $i$ in a jointly convex GNEP is given as, 
	\begin{equation}\label{strategyset}
		K_i(x_{-i})=\{x_i~\text{in}~ \mathbb{R}^{n_i}~|\,(x_{-i},x_i)~\text{lies in}~X\},
	\end{equation}
	where $X\subseteq \mathbb{R}^n$ is closed and convex. 
	
	Before the formal introduction of jointly convex GNEP in \cite{facchinei}, the early literature provided the characterization of classical GNEP with respect to a quasi-variational inequality (QVI) problem \cite{facc, faccbook}. But, this characterization is not useful enough because when compared to variational inequality (VI) problems, the literature of QVI is not equipped with well developed solving techniques (see \cite{faccbook} for more details). We recollect that for a given map $F:X \rightarrow \mathbb{R}^n$ (where $X\subseteq \mathbb{R}^n$), the variational inequality problem $VI(F,X)$ corresponds to determine $x^*\in X$ satisfying, 
	$$\langle F(x^*),y-x^*\rangle \geq 0~\text{for each}~y\in X.$$ 
	In the following result, Facchinei et al. \cite{facchinei} stated that under suitable conditions some solutions of jointly convex GNEP can be obtained by solving an associated VI problem,
	\begin{theorem}\cite{facchinei}\label{theoremFacc}
		Suppose $X\subset\mathbb{R}^n$ is non-empty, convex and closed. For any $i$, suppose the function $\theta_i:\mathbb{R}^{n-n_i}\times \mathbb{R}^
		{n_i}\rightarrow \mathbb{R}$ is continuously differentiable and $\theta_i(x_{-i},.)$ is a convex function for each $x_{-i}\in \mathbb{R}^{n-n_i}$. Then any $\bar x\in X$ solving $VI(F,X)$ is an equilibrium for jointly convex GNEP, where $F:\mathbb{R}^n\rightarrow \mathbb{R}^n$ is considered as $F(x)=\prod_{i=1}^{N}(\nabla_{x_i} \theta_i(x))$.
	\end{theorem}
	Here, $(\nabla_{x_i} \theta_i(x))$ denotes the partial derivative of the function $\theta_i$ w.r.t. variable $x_i$. According to Theorem \ref{theoremFacc}, it is not necessary that any equilibrium of jointly convex GNEP necessarily solves the associated variational inequality $VI(F,X)$. In fact, the equilibrium points for jointly convex GNEP which are obtained by solving $VI(F,X)$ are specifically termed as \textit{normalized Nash equilibrium} (NNE) \cite{facchinei, kanzow}. 
	Several iterative methods have been developed to find normalized Nash equilibrium for jointly convex GNEP over the past decade (see for e.g. \cite{facc, kanzow, pani, fukushima}). Therefore, investigating the conditions under which these iterative methods converge after finite number of steps is an important aspect.
	
	In the available literature, one can find various approaches related to finite termination of iterative methods used for solving optimization problems \cite{kypris, burke, polyak}, variational inequality problems \cite{patrikson,marcotte}, equilibrium problems \cite{moudafi, nguyen} and multi-criteria optimization \cite{bento}. 
	\textcolor{black}{In this direction, Burke-Ferris \cite{burke} studied the notion of weak sharp minima to investigate the finite termination of algorithms used for solving optimization problems with non-unique solution set. It extends the notion of sharp minimum which was studied in \cite{kypris, polyak} for the optimization problems having a unique solution. Later, Patricksson in \cite{patrikson} and Marcotte-Zhu in \cite{marcotte} studied the weak sharpness property for the solution set of variational inequalities in order to derive finite convergence results for iterative methods used in solving variational inequalities. On the other hand, Moudafi \cite{moudafi} developed $\theta$-conditioning criterion for equilibrium problems, which extends and unifies several existing concepts including weak sharp minima for optimization problem. Recently, Nguyen et al. \cite{nguyen} introduced the weak sharpness property for the solution set of equilibrium problems that is equivalent to the 1-conditioning (or linear conditioning) criterion. Moreover, the authors in \cite{moudafi,nguyen} demonstrated the finite termination of an iterative method employed for solving equilibrium problems.} As far as we know, there is no detailed study on the finite termination of iterative methods for solving GNEP using the weak sharpness property or linear conditioning criterion. 
	
	\textcolor{black}{In this work, we bring to the readers' attention two important aspects for the finite convergence analysis of jointly convex GNEP, including the weak sharpness property for the set of normalized Nash equilibria and the linear conditioning technique for regularized Nikaido-Isoda functions. In particular, we first give formal definition of the weak sharpness property for normalized Nash equilibria and then, we provide its characterization with respect to regularized gap function (used for determining NNE).} Furthermore, we provide the suitable conditions under which linear conditioning of regularized Nikaido-Isoda function is equivalent to the weak sharpness property for the set of normalized Nash equilibria. As an application, we obtain sufficient conditions under which any sequence calculated by employing a proximal point algorithm \cite{bento, flam} converges after finitely many iterations to an equilibrium solving the jointly convex GNEP. Finally, we apply our results to estimate the number of steps needed to determine NNE for a specific type of games consisting of $N$-players where the cost function for every player is assumed to be quadratic.
	
	This article is arranged in following sequence: Section \ref{prelim} recalls basic definitions and useful notations which we will use further. We introduce the weak sharpness property for the set of normalized Nash equilibria and give its characterization with respect to regularized gap function in Section \ref{weaksharp}. In Section \ref{linear}, we first provide the sufficient conditions under which linear conditioning of regularized Nikaido-Isoda function is equivalent to the weak sharpness property of the set of NNE. Further, in Section \ref{appl} we use this equivalence to show the finite termination of an iterative algorithm used for determining NNE. Finally, we employ our results to estimate the number of steps needed to determine the normalized Nash equilibrium for some specific type of games.
	
	\section{Preliminaries and Notations}\label{prelim}
	For any non-empty set $X\subseteq \mathbb{R}^n$, the set $X^{\circ}$ indicates polar set (refer \cite{rockfellar}),
	$$X^{\circ}=\{x^*\in \mathbb{R}^n|\,\langle x^*,z\rangle \leq 0~\text{for any}~z\in X\}.$$
	If $X$ is an empty set then $X^{\circ}=\mathbb{R}^n$. The set $N_X(x)$ indicates normal cone \cite{ rockfellar} corresponding to $X\subseteq \mathbb{R}^n$ at any $x\in \mathbb{R}^n$,
	$$N_X(x)= (X-x)^\circ=\begin{cases}
		\{x^*\in\mathbb{R}^n|\,\langle x^*,z-x\rangle \leq 0,~\forall~z\in X\},&\text{if}~x\in \mathbb{R}^n\\\emptyset,&\text{otherwise,}
	\end{cases}
	$$ and the tangent cone \cite{rockfellar} corresponding to $X$ at any $x$ is $T_X(x)=[N_X(x)]^\circ$. Alternatively,
	\begin{align*}
		T_X(x)=\{\,h\in\mathbb{R}^n|\,~\text{for any}~t_k \downarrow 0, \exists~h_k\rightarrow h~\text{s.t.}~x+t_k h_k\in X~ \text{for any}~ k\}. 
	\end{align*}
	
	Suppose $X\subseteq \mathbb{R}^n$ is non-empty closed convex then projection of any $z\in \mathbb{R}^n$ on $X$ is,
	$$P_X(z)=\{x\in X|\,\norm{z-x}=\inf_{w\in X}\norm{z-w}\}.$$
	
	A bifunction $\psi:\mathbb{R}^n\times \mathbb{R}^n\rightarrow\mathbb{R}$ is known as pseudomonotone \cite{nguyen,bigi} on $X\subseteq \mathbb{R}^n$ if we have $\psi(x,z) \leq 0$ whenever $\psi(z,x) \geq 0$ for each couple $x,z\in X$. Suppose $X^*\subset X$, then we call $\psi$ pseudomonotone on the set $X^*$ w.r.t. $X$ (renaming the monotonicity condition assumed in \cite[Definition 1.2]{flam}) if for any $z\in X^*$ and $x\in X$ we have $\psi(x,z)\leq 0$ whenever $\psi(z,x) \geq 0$. We call $\psi$ negatively psuedomonotone if $-\psi$ is psuedomontone.
	
	We have made the following assumptions about the feasible strategy map and loss function of each player $i$ throughout this article.
	\begin{itemize}
		\item[(A1)] the strategy map $K_i$ is considered as per (\ref{strategyset}), where $X\neq \emptyset$ is closed convex set in $\mathbb{R}^n$;
		\item[(A2)] the loss function $\theta_i:\mathbb{R}^{n-n_i}\times \mathbb{R}^
		{n_i}\rightarrow \mathbb{R}$ is continuously differentiable and $\theta_i$ is convex in $x_i$ for each $x_{-i}\in \mathbb{R}^{n-n_i}$.
	\end{itemize}
	Further, we suppose $X^*\subset X$ denotes the set of NNE for the given GNEP. 
	
	Let us recall the concept of \textit{Nikaido-Isoda} (NI) function $\psi:X\times X\rightarrow \mathbb{R}$ (see \cite{NI}),
	\begin{equation*}
		\psi(x,y)= \sum_{i=1}^{N} \big[\theta_i(x_{-i}, x_i)-\theta_i(x_{-i},y_i)\big].
	\end{equation*}
	A gap function $V:X\rightarrow \mathbb{R}$ for determining NNE (see \cite{facc}) can be constructed as follows by using NI function,
	\begin{equation}\label{V}
		V(x)= \sup_{y\in X} \psi(x,y).
	\end{equation}
	The authors in \cite{kanzow} noticed the fact that it is not necessary that the supremum in (\ref{V}) exists at unique $y\in X$ for any given $x\in X$, which results into non-differentiabilty of the gap function $V$. \textcolor{black}{Considering that $a$ is any positive real number, a regularized gap function $V_a$ is introduced in \cite{kanzow} for determining NNE, which becomes differentiable at each point in $X$ under certain conditions. The function $V_a$ is formed as $V_a(x)=\max_{y\in X} \psi_a(x,y)$, where $\psi_a:X\times X\rightarrow \mathbb{R}$ is regularized NI function \cite{kanzow}, }
	\begin{equation*}
		\psi_a(x,y)=\sum_{i=1}^{N} \bigg[\theta_i(x_{-i}, x_i)-\theta_i(x_{-i},y_i)-\frac{a}{2} \norm{x_i-y_i}^2\bigg].
	\end{equation*}
	
	Now, we provide the sufficient conditions under which regularized NI function $\psi_a$ becomes negatively pseudomonotone over the solution set $X^*$. In this regard, we have the below-mentioned result as a modification of \cite[Theorem 2.2]{flam}.
	\begin{lemma} \cite{flam}\label{pseudo}
		Let $X^*\subseteq X$ be a set of normalized Nash equilibria. Suppose that $\psi_a:X\times X\rightarrow \mathbb{R}$ is regularized NI function and $\psi_a(.,y)$ is convex function for any $y\in X$. Then,
		$$\psi_a (x,x^*) \geq 0,\enspace\text{for each}~x^*~\text{in}~ X^*~\text{and}~x~\text{in}~ X.$$ 
	\end{lemma}
	\begin{remark}
		One can deduce that $\psi_a$ is negatively pseudomonotone on the set $X^*$ w.r.t. $X$ if $\psi_a(.,y)$ is convex function for any $y\in X$. In fact, for each pair $x^*\in X^*, x\in X$ we already know that $\psi_a (x^*,x)\leq 0$ due to \cite[Definition 3.1]{kanzow}. Further, as per above lemma $\psi_a (x,x^*) \geq 0$ follows if we assume $\psi_a(.,y)$ is convex function for any $y\in X$.
	\end{remark}
	Suppose that $V_a:X\rightarrow \mathbb{R}$ is considered as $V_a(x)=\max_{y\in X}\psi_a(x,y)$ and $y^a(x)$ is the unique maximizer of $\psi_a(x,y)$ for given $x$ in $X$. \textcolor{black}{Then, the following result by Heusinger-Kanzow \cite{kanzow} states that $V_a$ is a differentiable regularized gap function for determining NNE if certain conditions hold.}
	\begin{lemma}\cite{kanzow}\label{lemma1}
		Suppose assumptions (A1) and (A2) holds. Then we have following consequences,
		\begin{itemize}
			\item[a)] $V_a(x)\geq 0$ for each $x\in X$;
			\item[b)] $x^*$ is a NNE iff $x^*\in X$ satisfies $V_a(x^*)=0$;
			\item[c)] $V_a$ being continuously differentiable yields, 
			\begin{align*}
				\nabla V_a(x)=&\sum_{i=1}^{N} \big[\nabla \theta_i(x_{-i}, x_i)-\nabla \theta_i(x_{-i}, y_i^a(x))\big ]-a(x-y^a(x))\\& + \big [\nabla_{x_1}\theta_1(x_{-1},y_1^a(x)),\cdots, \nabla_{x_N}\theta_N(x_{-N},y_N^a(x))\big ]^T.
			\end{align*}
		\end{itemize}
	\end{lemma}
	\section{Weakly Sharp Set of Normalized Nash Equilibria}\label{weaksharp}
	In \cite{burke}, Burke-Ferris initiated the notion of weak sharp minima to study the finite convergence of several iterative algorithms for convex optimization problem having non-unique solution. It extends the notion of sharp minima, which was studied in \cite{kypris, polyak} for the optimization problems admitting a unique solution. Suppose $f:\mathbb{R}^n\rightarrow \mathbb{R}\cup \{-\infty,\infty\}$ is an extended real-valued function. A set $Y^* \subset \mathbb{R}^n$ is known as set of weak sharp minima for the minimization problem $\min_{x\in Y} f(x)$ if there exists $\beta>0$ satisfying,
	\begin{equation}\label{def}
		f(x)\geq f(x^*)+\beta\,d(x,Y^*),\enspace\forall \,x\in Y~\text{and}~\forall \, x^*\in Y^*.
	\end{equation}

	Later, Pactricksson \cite{patrikson} and Marcotte-Zhu \cite{marcotte} studied the weak sharpness property for the solution set of VI problems. Consider $X\subseteq \mathbb{R}^n$ and a map $F:X \rightarrow \mathbb{R}^n$. Then variational inequality problem $VI(F,X)$ corresponds to determine $x^*\in X$ satisfying, 
	$$\langle F(x^*),y-x^*\rangle \geq 0~\text{for each}~y\in X.$$ According to \cite{patrikson}, the solution set $X^*$ of $VI(F,X)$ fulfills the weak sharpness property if there exists some positive real $\beta$ such that,
	\begin{equation} \label{patrick}
		-F(x^*) +\beta B\subset \bigcap\limits_{x\in X^*}[T_X(x)\cap N_{X^*}(x)]^\circ,~\forall\, x^*\in X^*,
	\end{equation}
	considering that $B$ indicates the unit ball in $\mathbb{R}^n$.
	
	\textcolor{black}{According to Theorem \ref{theoremFacc}, the set of normalized Nash equilibria for convex GNEP is nothing but the set of solutions for the variational inequality problem} $VI(F,X)$, where $F$ is defined as $F(x)=\prod_{i=1}^N({\nabla_{x_i} \theta_i}(x))$. Hence, we define the weak sharpness property for the set of normalized Nash equilibria by following Patricksson \cite{patrikson} and Marcotte-Zhu \cite{marcotte}.
	\begin{definition}
		Let $X^*\subseteq X$ denote the set of normalized Nash equilibria for a given convex $GNEP$. 
		Then the set $X^*$ fulfills weakly sharpness property if there exists some $\gamma >0$ such that,
		\begin{equation}\label{definition}
			-\prod_{i=1}^N({\nabla_{x_i} \theta_i(x^*)})+\gamma B\subseteq \bigcap\limits_{x\in X^*}[T_X(x)\cap N_{X^*}(x)]^\circ,~ \forall\, x^*\in X^*
		\end{equation}
		considering that $B$ indicates the unit ball in $\mathbb{R}^n$.
	\end{definition}

	In \cite{marcotte}, Marcotte-Zhu provided the suitable conditions under which the solution set of $VI(F,X)$ fulfills the weak sharpness property iff there exists some $\beta>0$ satisfying,
	\begin{equation}\label{dual}
		G(y)\geq \beta d(y,X^*),~\text{for any}~y\in X
	\end{equation}
	considering that $G$ denotes dual gap function,
	\begin{equation*}
		G(y)=\sup_{w\in X}\langle F(w),y-w\rangle.
	\end{equation*}
	
	The above-mentioned inequality (\ref{dual}) due to Marcotte-Zhu \cite{marcotte} motivated us to investigate the necessary and sufficient conditions under which the weak sharpness property for the set of NNE can be characterized in terms of a regularized gap function $V_a$ (see Lemma \ref{lemma1} for definition of $V_a$). 
	\begin{theorem}\label{theorem1}
		\textcolor{black}{Suppose that $\psi_a:X\times X\rightarrow \mathbb{R}$ indicates regularized NI function and $\psi_a(.,y)$ is a convex function for any $y\in X$.} Then, the set of normalized Nash equilibria $X^*\subseteq X$ fulfills the weak sharpness property iff there exists $\gamma >0$ satisfying,
		\begin{equation}\label{result}
			\tag{A}
			V_a(x)\geq \gamma \, d(x,X^*)~\text{for any}~x\in X.
		\end{equation}
	\end{theorem}
	\begin{proof}
		Consider a function $F:\mathbb{R}^n\rightarrow \mathbb{R}^n$ is formed as $F(x)=\prod_{i=1}^N(\nabla_{x_i} \theta_i(x))$. 
		We aim to prove that (\ref{result}) holds if the set $X^*$ consisting of NNE fulfills the weak sharpness property. According to (\ref{definition}), we obtain some $\gamma>0$ satisfying, 
		\begin{equation}\label{eq4}
			-F(x^*)+\gamma b\in [T_X(x^*)\cap N_{X^*}(x^*)]^\circ
		\end{equation}
		for each $b ~\text{in}~ B$ and $x^*~\text{in}~ X^*$.
		
		Suppose $x^*~\text{in}~ X^*$ is chosen arbitrarily, then one can show that inclusion (\ref{eq4}) holds for $x^*$ iff,
		\begin{equation}\label{eq5}
			\langle F(x^*),z\rangle \geq \gamma \norm{z}~\text{for each}~z\in T_X(x^*)\cap N_{X^*}(x^*).
		\end{equation}
		Clearly, the following inequality can be obtained for each $b\in B$ if it is given that inclusion (\ref{eq4}) holds,
		$$\langle -F(x^*)+\gamma b,z\rangle \leq 0~\text{for each}~z\in T_X(x^*)\cap N_{X^*}(x^*)$$
		Now, one can easily deduce (\ref{eq5}) by taking $b=\frac{z}{\norm{z}}$.
		Conversely, suppose inequality (\ref{eq5}) holds. Then for all $b$ in $B$ and $z$ in  $T_X(x^*)\cap N_{X^*}(x^*)$, it occurs
		\begin{align*}
			\langle-F(x^*) +\gamma b,z\rangle
			\leq -\langle F(x^*),z\rangle +\gamma\norm{z} \leq 0
		\end{align*}
		which shows $-F(x^*)+\gamma b\in [T_X(x^*)\cap N_{X^*}(x^*)]^\circ$ for any $b\in B$.
		
		Now, by following \cite{kanzow} one can easily verify that $V_a$ is a convex  function if $\psi_a(.,y)$ is a convex function for each $y$ in $X$. In fact, for each $x$ in $X$ we always get a unique maximizer since $\psi_a(x,.)$ is strictly concave. Assume that $y^a(x)$ denotes the unique maximizer of $\psi_a(x,.)$ for $x$ in $X$. Suppose $x,w\in X$ and $t\in[0,1]$ are arbitrary, then by assuming $tx+(1-t)w=w_t$ we obtain,
		\begin{align*}
			V_a(w_t)&=\displaystyle \max_{y\in X}\psi_a(w_t,y)\\
			&=\psi_a(w_t,y^a(w_t))\\
			&\leq t\psi_a(x,y^a(w_t))+(1-t)\psi_a(w,y^a(w_t))\\
			&\leq t\psi_a(x,y^a(x))+(1-t)\psi_a(w,y^a(w))\\
			&=tV_a(x)+(1-t)V_a(w).
		\end{align*}
		
		Since $V_a$ is a continuously differentiable function as per Lemma \ref{lemma1}, we observe $\nabla V_a(x^*)=F(x^*)$ for any $x^*\in X^*$. In fact, due to Lemma \ref{lemma1} we already have,
		\begin{equation}\label{derivative}
			\nabla V_a(x)=\sum_{i=1}^{N} [\nabla \theta_i(x_{-i}, x_i)- \nabla \theta_i(x_{-i}, y_i^a(x))]+[\nabla_{x_i}\theta_i(x_{-i},y_i^a(x))]_{i=1}^{N}-a(x-y^a(x))
		\end{equation}
		for any $x\in X$. According to \cite[Prop. 3.4]{kanzow}, a point $x^*$ lies in a set of NNE $X^*$ iff $x^*=y^a(x^*)$. Thus, by equation (\ref{derivative}) we obtain,
		\begin{equation}\label{deri}
			F(x^*)=\nabla V_a(x^*)~\text{for each}~ x^*~\text{in}~ X^*.
		\end{equation}
		
		Now, suppose $x\in X$ be arbitrary. Since the set $X^*$ is closed convex, a unique point $x^*\in X^*$ exists which satisfies $x^*=P_X(x)$ and $(x-x^*)\in T_X(x^*)\cap N_{X^*}(x^*)$. 
		Finally, by using to convexity and differentiability of the function $V_a$ we get,
		\begin{align*}
			V_a(x)&\geq V_a(x^*)+\langle \nabla V_a(x^*),x-x^*\rangle \\
			&=\langle F(x^*),x-x^*\rangle \\
			&\geq \, \gamma \norm{x-x^*}=\gamma d(x,X^*).
		\end{align*}
		Since $V_a(x^*)=0$ (see Lemma \ref{lemma1}), the equality in second step holds due to (\ref{deri}) and last inequality holds due to (\ref{eq5}).
		
		Conversely, suppose that there is some $\gamma>0$ such that the given inequality (\ref{result}) holds for the set of NNE $X^*$. Then, we aim to show that $X^*$ satisfies (\ref{patrick}). Taking $x^*\in X^*$ arbitrarily 
		we observe that $T_X(x^*)\cap N_{X^*}(x^*)=\{0\}$ implies $-F(x^*)+\gamma B \subset [T_X(x^*)\cap N_{X^*}(x^*)]^\circ$. Suppose some non-zero vector $s$ is in $T_X(x^*)\cap N_{X^*}(x^*)$ then $\langle s,s\rangle >0$ and hence, $x^*+s\notin X^*$. This implies, there exists a hyperplane $\mathcal{H}_s$ containing the point $x^*$ and orthogonal to $s$ which separates $x^*+s$ from $X^*$.
		Since $s\in T_X(x^*)$, we obtain a sequence $s_k$ converging to $s$ for any sequence $t_k \downarrow 0$ such that $x^*+t_k s_k\in X$. Furthermore, one can obtain a $\tilde k \in \mathbb{N}$ so that $\mathcal{H}_s$ separates $x^*+t_k s_k$ from $X^*$ for each $k>\tilde k$. Hence, 
		\begin{align}\label{disthyper}
			d(x^*+t_k s_k,X^*)\geq d(x^*+t_k s_k,\mathcal{H}_s)
			=\frac{t_k\langle s,s_k\rangle}{\norm{s}},\enspace\forall~k>\tilde k.
		\end{align}
		Since $V_a(x)\geq \gamma d(x,X^*)$ for each $x$ in $X$ and $V_a(x^*)=0$ for each $x^*$ in $X^*$, we observe that following condition holds in the view of inequality (\ref{disthyper}), 
		\begin{align*}
			\frac{V_a(x^*+t_k s_k)-V(x^*)}{t_k}\geq\gamma \frac{\langle s,s_k\rangle}{\norm{s}},\enspace\forall~k>\tilde k.
		\end{align*}
		Now, we obtain $\langle\nabla V_a(x^*),s\rangle \geq \gamma \norm{s}~\text{for any}$ $0\neq s\in T_X(x^*)\cap N_{X^*}(x^*)$ by using differentiability of the function $V_a$.
		Eventually, $\langle\gamma b-\nabla V_a(x^*),s\rangle\leq 0$ is fulfilled for any $b\in B, s\in T_X(x^*)\cap N_{X^*}(x^*)$. This fact, together with equation (\ref{deri}) implies, 
		\begin{equation} \label{Fx}
			-F(x^*)+ \gamma B\subseteq [T_X(x^*) \cap N_{X^*}(x^*)]^\circ~ \text{for each}~ x^*~\text{in}~ X^*.
		\end{equation}
		
		Since the function $V_a$ is convex and differentiable, we notice $\nabla V_a$ is constant over $X^*$ as per \cite[Lemma 1]{zhou}. Hence, $\nabla V_a(x^*)=C=F(x^*)$ for each $x^*$ in $X^*$ for some $C\in \mathbb{R}$. Finally, one can observe that the fact $F$ is constant over $X^*$ along with the inclusion (\ref{Fx}) yields (\ref{definition}).
		
	\end{proof}
	\begin{remark}
		We observe that Marcotte-Zhu \cite{marcotte} characterized the weakly sharp solution set of VI problems by using a dual gap function. Furthermore, Liu-Wu \cite{liu} studied the similar type of characterization in terms of primal gap function by assuming that the primal gap function is Gateaux differentiable. We provided the error bound in terms of regularized gap function used for determining NNE, which is differentiable according to \cite{kanzow}. In the similar way, one can obtain such type of error bound without any differentiability assumption by considering the regularized primal gap function of VI problems.
	\end{remark}
	
	\section{Linearly Conditioned Regularized Nikaido- Isoda function}\label{linear}
	Moudafi initiated notion of $\theta$-conditioned bifunction in \cite{moudafi}. Using this concept, the author shown that under some suitable conditions the proximal point method converges for equilibrium problem after finite number of steps. Let $T:X\times X\rightarrow \mathbb{R}$ be a bifunction defined on a set $X\subseteq \mathbb{R}^n$. An equilibrium problem corresponds to determine some $x^*\in X$ s.t.,
	$$T (x^*,z)\geq 0,~\text{for every}~z\in X.$$
	Suppose $X^*$ indicates set of solutions for above equilibrium problem. Then, $T$ is known as $\theta$-conditioned \cite{moudafi} for some $\theta >0$, if there is $\gamma >0$ fulfilling,
	$$ -T(x,P_{X^*}(x)) \geq \gamma [d(x, X^*)]^\theta,~\text{for every}~x\in X. $$
	Furthermore, $T$ is said to be linearly conditioned if $\theta=1$.
	
	According to \cite[Definition 3.1]{kanzow}, we know that any point $x^*\in X$ is NNE if,
	\begin{equation}\label{defNI}
		-\psi_a (x^*,y)\geq 0,~\text{for every}~y\in X.
	\end{equation}
	Therefore, regularized NI function becomes linearly conditioned if we have some $\gamma >0$ fulfilling,
	\begin{equation}
		\psi_a (x, P_{X^*}(x))\geq \gamma d(x,X^*),~\text{for every}~x\in X.\label{defn}
	\end{equation} 
	
	\subsection{Relation between Linear Conditioning of Regaularized NI function and Weak Sharpness property for the set of NNE} \hfill \\
	
	\textcolor{black}{In following result, we state the suitable conditions that ensure the linear conditioning} of regaularized Nikaido-Isoda function \cite{moudafi} is equivalent to the weak sharpness property for the set of NNE. 
	\begin{theorem} \label{theoremequi}
		\textcolor{black}{Suppose that $\psi_a:X\times X\rightarrow \mathbb{R}$ indicates regularized NI function and $\psi_a(.,y)$ is a convex function for any $y\in X$.} Then $X^*\subset X$ consisting of NNE fulfills weakly sharpness property iff $\psi_a$ is linearly conditioned. 
	\end{theorem}
	\begin{proof}
		Assume that $X^*$ fulfills the weak sharpness property. Then by (\ref{definition}),
		$$ \gamma B-F(x^*) \subseteq [T_X (x^*)\cap N_{X^*} (x^*)]^\circ, ~\text{for each}~x^*~\text{in}~  X^*.$$
		As per proof of Theorem \ref{theorem1}, one may observe that $F(x^*)= \nabla V_a (x^*)$ for each $x^*\in X^*$. Hence, we obtain
		\begin{equation} \label{eqA}
			\gamma B-\nabla V_a (x^*) \subseteq [T_X (x^*)\cap N_{X^*} (x^*)]^\circ, ~\text{for each}~x^*~\text{in}~ X^*.
		\end{equation}
		Let $x$ be arbitrary point in $X$. Since $X^*$ is a closed and convex set, a unique point $x^*\in X^*$ with $P_{X^*}(x)= x^*$ can be obtained. 
		
		We now claim that $\psi_a$ is linearly conditioned. In fact, on considering $P_{X^*}(x)=x$ this claim follows trivially. Suppose $P_{X^*}(x)\neq x$, then inclusion (\ref{eqA}) along with the fact $(x-x^*)\in T_X(x^*)\cap N_{X^*}(x^*)$ implies,
		$$ \Big \langle \frac{\gamma (x-x^*)}{\norm {x-x^*}}-\nabla V_a (x^*), x-x^*\Big \rangle \leq 0.$$
		This further results into,
		\begin{equation} \label{eqB}
			\gamma \norm{x-x^*}- \langle \nabla V_a (x^*), x-x^*\rangle \leq 0.
		\end{equation} 
		As per \cite[Prop. 3.4]{kanzow}, any point $x^*$ is in $X^*$ iff $x^*=y_a (x^*)$ where $y_a (x^*)$ solves the following maximization problem,
		$$ V_a (x^*)=\max_{y\in X} \psi_a (x^*,y).$$
		Thus, by applying Danskin's theorem \cite{faccbook} for any $x^*$ in $X^*$ it occurs,
		$$ \nabla V_a (x^*)= \nabla_x [\psi_a (x^*, y)]_{(y=y_a(x^*))}$$ 
		where $\nabla_x \psi_a (x^*,y)$ denotes the partial derivative of the function $\psi_a$ w.r.t. $x$ at the point $x=x^*$. Now, we have following due to inequality (\ref{eqB}),
		\begin{equation} \label{eqC}
			\gamma d(x,X^*) \leq \langle \nabla_x \psi_a (x^*,x^*), x-x^*\rangle.
		\end{equation} 
		Since $\psi_a (.,y)$ is convex for every $y\in X$ and $\psi_a(x^*,x^*)=0$, it occurs,
		\begin{align}\label{eqD}
			\langle \nabla_x \psi_a (x^*,x^*), x-x^*\rangle  &\leq \psi_a (x,x^*).
		\end{align}
		Finally, inequality (\ref{eqC}) along with (\ref{eqD}) yields,
		$$ \psi_a (x,P_{X^*}(x)) \geq \gamma  d(x,X^*).$$
		
		Conversely, we assume $\psi_a$ is linearly conditioned. We prove that $X^*$ fulfills the weak sharpness criterion. Suppose $x\in X$ is arbitrary. Since for each $x$ in $X$, a vector $y_a (x)$ exists uniquely which solves,
		$$ V_a (x)= \max_{y\in X} \psi_a (x,y).$$
		Thus, $V_a(x)=\psi_a (x, y_a (x)) \geq \psi_a (x, P_{X^*}(x)) \geq \gamma d(x,X^*)$ due to (\ref{defn}).
		Finally, by using Theorem \ref{theorem1} we affirm that $X^*$ fulfills the weak sharpness condition. 
		
	\end{proof}
	
	\begin{remark}
		The authors in \cite{nguyen} considered an equilibrium problem involving a bifunction $F$. According to \cite[Remark 3]{nguyen}, the monotonicity of the bifunction $F$ is required in order to ensure that $F$ is linear conditioned if the solution set of equilibrium problem fulfills the weak sharpness condition. However, in Theorem \ref{theoremequi} we proved that linear conditioning of regularized NI function $\psi_a$ is equivalent to the weak sharpness property for the set of NNE, where $\psi_a$ is not necessarily monotone. In fact, the regularized NI function $\psi_a$ is not even pseudomonotone on the set $X$. According to Lemma \ref{pseudo}, negatively pseudomonotone over the solution set $X^*$ w.r.t $X$ (refer \cite{bigi} for more information on the various monotonicity properties of bifunctions). 
	\end{remark}
	\section{Finite Termination of an Iterative Algorithm}\label{appl}
	Now, we demonstrate the termination of proximal point algorithm \cite{moudafi,flam} after finite number of steps if the set of NNE fulfills the weak sharpness property (or equivalently regularized NI function is linearly conditioned). For this purpose, we apply the results derived in previous sections. 
	
	As per (\ref{defNI}), one can employ the proximal point algorithm (PPA) for determining the set of NNE. Let $x_0\in X$ be chosen arbitrarily, then the sequence $\{x_k\}_{k\in\mathbb{N}}$ generated by PPA is given as \cite{moudafi,flam},
	\begin{equation}\label{proximal}
		\psi_a (x_{k+1},z)-\frac{1}{r_k} \langle z-x_{k+1}, x_{k+1}-x_k\rangle \leq 0,~\text{for each}~z\in X,  
	\end{equation}
	where the sequence $\{r_k\}_{k\in\mathbb{N}}$ fulfills $r_k>0~\text{for each}~k$ and $\liminf_{k\rightarrow \infty} r_k=r>0$.
	\begin{theorem}\label{finite}
		\textcolor{black}{Suppose that $\psi_a(.,y)$ is a convex function for any $y\in X$. Then, for each sequence $\{x_k\}_{k\in\mathbb{N}}$ generated by (\ref{proximal}), some $k_\circ\in \mathbb{N}$ can be obtained which satisfies $x_k \in X^*$ for all $k> k_\circ$ if one of the following condition holds,}
		\begin{itemize}
			\item[(a)] the regularized NI function $\psi_a$ is linearly conditioned;
			\item[(b)] the set of normalized Nash equilibria $X^*$ fulfills the weak sharpness property.
		\end{itemize}
	\end{theorem}
	\begin{proof}
		Suppose hypothesis (a) holds. Since $\psi_a$ is linearly conditioned, we obtain some $\gamma>0$ satisfying,
		\begin{equation} \label{eqH}
			\psi_a (x, P_{X^*}(x)) \geq \gamma d(x, X^*).
		\end{equation} 
		
		In order to derive the stated result, we first prove that any sequence $\{x_k\}_{k\in\mathbb{N}}$ generated by PPA (\ref{proximal}) satisfies $\norm{x_k-x_{k+1}}\rightarrow 0$ as $k\rightarrow \infty$. Suppose $z\in X$ is arbitrary then the expression in (\ref{proximal}) can be written as,
		\begin{align*}
			\psi_a (x_{k+1},z)+\frac{1}{2r_k}  \norm{x_k-x_{k+1}}^2+\frac{1}{2r_k} \norm{z-x_{k+1}}^2- \frac{1}{2r_k} \norm{z-x_k}^2 \leq 0.
		\end{align*}
		Suppose $x^*\in X^*$ is arbitrary. Then, by considering $z=x^*$ we get,
		\begin{align} \label{prox1}
			\psi_a (x_{k+1},x^*)+\frac{1}{2r_k} \norm{x_k-x_{k+1}}^2+\frac{1}{2r_k} \norm{x^*-x_{k+1}}^2-\frac{1}{2r_k} \norm{x^*-x_{k}}^2 \leq 0.
		\end{align}
		Since $x^*\in X^*$ is a NNE, we observe $\psi_a(x^*,z)\leq 0$ for any $z$ in $X$ by virtue of (\ref{defNI}). Eventually, we obtain $\psi_a (z,x^*)\geq 0$ for any $z$ in $X$ due to Lemma \ref{pseudo}. Therefore, inequality (\ref{prox1}) yields,
		\begin{equation} \label{ine}
			\norm{x_{k+1}-x_k}^2\leq  \norm{x^*-x_{k}}^2-\norm{x^*-x_{k+1}}^2.
		\end{equation} 
		This further results into,
		$$\sum_{k=0}^\infty \norm{x_{k+1}-x_k}^2\leq \norm{x^*-x_0}^2<\infty.$$
		Hence, $\norm{x_{k+1}-x_k}\rightarrow 0$ as $k\rightarrow \infty$.
		
		We observe that the sequence $\{\frac {1}{r_k}\}_{k\in\mathbb{N}}$ is a bounded sequence as $\liminf_{k\rightarrow \infty} r_k=r>0$. Hence, $\frac{1}{r_k} \norm{x_{k+1}-x_k}\rightarrow 0$ as $k\rightarrow \infty$. Equivalently, for $\gamma >0$ some $k_\circ \in \mathbb{N}$ exists satisfying,
		\begin{equation}\label{eqJ}
			\frac{1}{r_k} \norm{x_{k+1}-x_k} <\gamma ~\,\text{for each}~k\geq k_\circ.
		\end{equation}
		We aim to show that $x_k\in X^*$ for each $k>k_\circ$. Suppose, on contrary $x_{j}\notin X^*$ for some $j>k_\circ$ then $d(x_{j},X^*)\neq 0$ (as $X^*\subset X$ is closed convex). 
		Hence, following occurs due to (\ref{proximal}),
		\begin{equation}
			\psi_a (x_{j}, P_{X^*}(x_{j}))\leq \frac{1}{r_{j-1}} d(x_{j},X^*) \norm{x_{j}-x_{j-1}} .\label{eqI}
		\end{equation} 
		We obtain following inequality on combining (\ref{eqH}) and (\ref{eqI}),
		$$ \norm{x_{j}-x_{j-1}} \geq \gamma r_{j-1}.$$
		But, this contradicts (\ref{eqJ}) as $j-1\geq k_\circ$. Eventually, $x_k$ is in $X^*$ for all $k> k_\circ$.
		
		According to Theorem \ref{theoremequi}, the regularized NI function $\psi_a$ is linearly conditioned iff $X^*$ fulfills the weak sharpness property. Hence, the proof follows in case hypothesis (b) holds.
		
	\end{proof}
	\subsection{Application to the Games with Quadratic Loss Functions}
	\hfill \\
	
	Let us consider a jointly convex GNEP in which the loss function for each player is defined with the help of matrices. We apply the results obtained in previous sections to obtain an estimate for the number of iterations required to determine the normalized Nash equilibria for such games. We prove the below-stated result by following some arguments presented in \cite{nguyen}.
	\begin{theorem}
		Suppose,
		\begin{align*}
			\theta_i(x)=\frac{1}{2} (x_i)^T A_{ii} x_i+\sum_{l=1,l\neq i}^{N} (x_l)^T A_{li} x_i~\forall i=1,2,\cdots N.
		\end{align*}
		where $~x=(x_i)_{i=1}^N \in \mathbb{R}^n$, $A_{li}$ are matrices in $\mathbb{R}^{n_l\times n_i}$ and $A_{ii}$ are symmetric matrices with $\sum_{i=1}^N n_i=n$. Let us consider a positive definite matrix $C$ defined as,
		$$ C=\begin{bmatrix}
			\frac{1}{2} A_{11} & A_{12} & \cdots A_{1N}\\
			A_{21} & \frac{1}{2} A_{22} & \cdots A_{2N}\\
			\vdots\\
			A_{N1} & A_{N2} & \cdots \frac{1}{2} A_{NN}\\
		\end{bmatrix}.$$
		Suppose $\delta >0$ denotes the smallest eigen value for $C+C^T$. If the set of normalized Nash equilibria fulfills the weak sharpness condition, then by considering $a\in (0,\delta]$ and $\frac{1}{r_k}\in (0,\epsilon)$ for some $\epsilon>0$, we observe that the PPA in (\ref{proximal}) converges in $k_\circ+1$ iterations for some initial point $x_\circ$ in $X$ where, 
		$$ k_\circ \leq \frac{{d(x_\circ, X^*)}^2 \epsilon ^2}{\gamma^2}.$$
	\end{theorem}
	\begin{proof}
		According to \cite[Proposition 3.1.2]{heusthesis}, one can observe that the regularized NI function $\psi_a(x,y):\mathbb{R}^n\times \mathbb{R}^n\rightarrow \mathbb{R}$ corresponding to the above-mentioned GNEP is convex w.r.t. $x$ for any $y\in X$. Since $X^*$ is assumed to satisfy the weak sharpness property, we observe that $\psi_a$ is linearly conditioned as per Theorem \ref{theoremequi}. Now, according to Theorem \ref{finite} we obtain some $k_\circ \in \mathbb{N}$ satisfying $x_k\in X^*$ for each $k>k_\circ$ where $\{x_k\}$ is a sequence generated by PPA in (\ref{proximal}). 
		
		By following the proof of Theorem \ref{finite}, the reader may observe that $\norm{x_{k+1}-x_k}\rightarrow 0$ as $k\rightarrow\infty$. Hence, we get some least positive $k_\circ \in \mathbb{N}$ satisfying,
		\begin{equation}\label{ineq}
			\norm{x_{k+1}-x_{k}}< \frac{\gamma}{\epsilon} ~\text{for all}~ k\geq k_\circ.
		\end{equation}
		We claim that $x_k\in X^*$ for each $k>k_\circ$. Suppose, on contrary $x_{j}\notin X^*$ for some $j>k_\circ$. Let $P_{X^*}(x_{j})=y_{j}$ then by linear conditioning and (\ref{proximal}) we obtain,
		\begin{align*}
			\gamma \norm{x_{j}-y_{j}}&=\gamma d(x_{j},X^*)\leq \psi_a(x_{j},y_{j})\\
			&\leq  \frac{1}{r_{j}} \langle y_{j}-x_{j}, x_{j}-x_{j-1}\rangle \\
			&< \gamma \norm{y_{j}-x_{j}}.
		\end{align*}
		Clearly, this contradiction yields $x_k\in X^*$ for any $k>k_\circ$. 
		
		Now, one can easily obtain following from inequality (\ref{ine}),
		\begin{align*}
			\norm{x^*-x_\circ}^2 
			&\geq \norm{x^*-x_1}^2+\norm{x_1-x_\circ}^2\\
			&\geq \cdots\\
			&\geq \norm{x^*-x_{k_\circ}}^2+ \sum_{v=0}^{k_\circ-1} \norm{x_{v+1}-x_v}^2\\
			&\geq \sum_{v=0}^{k_\circ-1} \norm{x_{v+1}-x_v}^2
		\end{align*}
		Consequently, we get following by employing inequality (\ref{ineq}), $$d(x_\circ, X^*)^2=\inf_{x^*\in X^*} \norm{x^*-x_\circ}^2\geq \displaystyle\sum_{v=0}^{k_\circ-1} \norm{x_{v+1}-x_v}^2 \geq \frac{(k_\circ) \gamma^2}{\epsilon^2},$$ 
		which gives us the required estimate,
		$$ k_\circ \leq \frac{d(x_\circ, X^*)^2 \epsilon^2}{\gamma^2}.$$ 
	\end{proof}
	\vskip 6mm
	\noindent{\bf Acknowledgements}
	
	\noindent The first author acknowledges Science and Engineering Research Board (SERB), India for the financial support under $(\rm{MTR/2021/000164})$. The second author is grateful to the University Grants Commission (UGC), New Delhi, India for the financial assistance provided by them throughout this research work under the grant $(\rm{1313/(CSIRNETJUNE2019)})$.
	
\end{document}